\documentclass[10pt]{amsart}
\usepackage{amssymb}
\usepackage{tikz}
\usepackage{latexsym}
\usepackage{epsfig}
\usepackage{graphicx}
\usepackage{psfrag}
\usepackage{ amssymb, amsmath}
\usepackage{amsfonts}
\usepackage{amsthm}
\usepackage{enumerate}
\usepackage{color}
\newtheorem{theorem}{Theorem}[section]

\newtheorem{lemma}[theorem]{Lemma}
\newtheorem{proposition}[theorem]{Proposition}
\newtheorem{observation}[theorem]{Observation}

\newcommand{\what}{\widehat}%
\newcommand{\wtilde}{\widetilde}%
\newcommand{\R}{\mathbb R}%
\newcommand{\C}{\mathbb C}%
\newcommand{\Z}{\mathbb Z}%
\newcommand{\N}{\mathbb N}%
\newcommand{\Hq}{\mathbb H}%
\newcommand{\hc}{\mathrm c}
\newcommand{\Ab}{\mathcal A}
\numberwithin{equation}{section}
\begin{document}
\baselineskip15pt
\author[R. P. Sarkar]{Rudra P. Sarkar}
\address[R. P. Sarkar]{Stat-Math Unit, Indian Statistical
Institute, 203 B. T. Rd., Calcutta 700108, India, E-mail:
rudra@isical.ac.in}

\title[Fourier transform supported on annulus, eigenfunction]{Function with its Fourier transform  supported on  annulus and eigenfunction of Laplacian}

\maketitle
\begin{abstract} We explore  the possibilities of reaching the characterization of eigenfunction of Laplacian as a degenerate case of the inverse Paley-Wiener theorem (characterizing functions whose Fourier transform is supported on a compact annulus) for  the Riemannian symmetric spaces of noncompact type. Most distinguished prototypes of these spaces are the hyperbolic spaces.
The statement and the proof of the main result work mutatis-mutandis for a number of spaces including Euclidean spaces and Damek-Ricci spaces.
\end{abstract}
\section{Introduction}
Let $X$ be a rank one Riemannian  symmetric space  of noncompact type of dimension  $d$, $\Delta$ be the Laplace-Beltrami operator on $X$ induced by its Riemannian structure and $B$ be its  maximal distinguished boundary which is diffeomorphic to  $\mathbb S^{d-1}$. A prototypical  example of this class of spaces is the hyperbolic $n$-space. A representative result of this note  is the following.
\begin{theorem} \label{main-result}
Suppose that for a nonzero function   $f\in L^{2, \infty}(X)$,  there are constants  $c_1\ge \rho^2, c_2 \le 1/\rho^2$ such that
\[\lim_{n\to \infty} \|\Delta^n f\|_{2, \infty}^{1/n}=c_1,\,\, \lim_{n\to \infty} \|\Delta^{-n} f\|_{2, \infty}^{1/n}=c_2.\]
Let $\beta=\sqrt{1/c_2-\rho^2} \text{ and } \alpha=\sqrt{c_1-\rho^2}$. Then we have the following conclusions.
\begin{enumerate}[{\em (a)}]
\item $c_1 c_2\ge 1$.
\item If $c_1 c_2>1$ then $\wtilde{f}$ is supported in the annulus $\mathbb A_\beta^\alpha=[\beta, \alpha]\times  B$ around origin, but not
inside any smaller annulus $\mathbb A_{\beta'}^{\alpha'}$ where $\beta<\beta'$ or $\alpha'<\alpha$.
\item If $c_1c_2=1$ then $f=\mathcal P_\alpha F$ for some $F\in L^2(B)$, which is an eigenfunction of $\Delta$  with eigenvalue $-c_1$.
\item The annulus $\mathbb A_\beta^\alpha$ containing support of $\wtilde{f}$ may reduce to  a ball around origin: $\mathbb A_0^\alpha=[0, \alpha]\times B$, but cannot collapse  to the  origin.
\end{enumerate}
\end{theorem}
(See Theorem \ref{towards-distribution-result} for a generalization and Proposition \ref{gen-eigen-for-nonradial} for a related result of independent interest.)

For a Schwartz class function $g$ on $X$, $\wtilde{g}$ is an analogue of the Fourier transform on $\R^n$ in polar coordinates and is known as  geometric or the Helgason Fourier transform,   defined on $\R^+\times B$. In the statement  above $\wtilde{f}$ is taken in the sense of tempered distribution. The Poisson transform $\mathcal P_\alpha$  is an analogue of the operator $P_\lambda$ given by $P_\lambda F(x)=\int_{\mathbb S^{n-1}} F(y) e^{i\lambda x.y} dy$ on $\R^n$. While  $P_\lambda$  maps a suitable function $F$ on the  {\em boundary} $\mathbb S^{n-1}$  of $\R^n$  to a  function on $\R^{n}$, $\mathcal P_\alpha$ maps a function $F$ defined on $B$ to a function on $X$. Indeed $P_\lambda F$ or $P_\alpha F$ are the basic eigenfunctions of the Laplacian of the corresponding spaces.   In the hypothesis   $\Delta^n f$ is used  in the sense of distribution while $\Delta^{-n}f$ is in the sense of multiplier as spectrum of $\Delta$ on $X$ is $[-\infty, -\rho^2]$  where $\rho$, the half-sum of positive roots, is realized as  a positive number. We keep away from  these interpretational worries, as we shall discuss them in details  in Section \ref{preparation}.  For other  notation see Section 2.

An analogue of this theorem can be proved for $\R^n$ replacing $L^{2, \infty}$-norm   by $L^\infty$-norm or by $L^p$-norm with $p>2n/(n-1)$ for $n>1$, to ensure the possibility of  accommodating  the eigenfunctions of Laplacian. One  also obtains  an analogue for the Heisenberg groups $\Hq_n$  with $L^\infty$-norm in the hypothesis, and using the ``Fourier transform''  as defined in \cite{Str-Roe}. See \cite{Gabardo1989, Howd-Reese, Str-Roe}, where some parts of these results for $\R^n$ and $\Hq_n$ are implicit. The situation in the Riemannian symmetric spaces of noncompact type, appears to be more intriguing, as indicated   in \cite{Str-Roe} by  constructing a counter example of Euclidean result for a complex hyperbolic space.  We may point out here that the choice of the weak $L^2$-norm (i.e. $L^{2, \infty}$-norm) in the hypothesis is not at all arbitrary. Indeed, among  all the Lorentz-norms (which, we recall include all $L^p$-norms), $L^{2, \infty}$-norm is the unique option for $X$  through which the theorem can accommodate   the two possibilities (b) and (c) about the function $f$.  We shall elaborate on these in  Section \ref{preparation} and cite  some other ``close to $L^2$'' norms which can be used in place of weak $L^2$-norm.
Theorem \ref{main-result} and its proof  extends to the Damek-Ricci spaces (also called  $NA$ groups) which are Riemannian manifolds and solvable Lie groups but not in general symmetric spaces. Indeed  rank one symmetric spaces of noncompact type accounts for a very small subset of all $NA$ groups.
However, we choose to  illustrate the phenomenon only  on rank one symmetric space, since, extending this to the  set up of $NA$ groups requires a lot of preliminaries, but the proof turns out to be the same (see Section 5 (3)).

To orient the readers  we shall  add some perspective of this study.
An inverse Paley-Wiener theorem gives criterion on a function (with some integrability or regularity) which is necessary and sufficient  for its Fourier transform to be compactly supported in a ball around origin, through the  holomorphic extension of the function along with a growth condition on it. For Euclidean spaces it is same as the usual Paley-Wiener theorem. But for other spaces (e.g.  a semisimple or nilpotent Lie group or a symmetric space) where it is plausible to talk about Fourier transform,  the usual and the inverse Paley-Wiener theorems are distinguished by the fact that domain of a function and its Fourier transform  may be quite different and it is not at all clear where the complex analytic extension of the function has to be considered. For non-Euclidean spaces such inverse Paley-Wiener theorems are rather recent
(see e.g. \cite{Faraut, Pasquale, Kroz-Stanton}). Very  roughly,  they state  again  that a suitable function with its Fourier transform compactly supported on its domain  can be characterized from  the holomorphic extension (in an appropriate domain) and  growth of the function.

Unlike these results a {\em real} inverse  Paley-Wiener theorem,   does not consider and use  the  holomorphic extension of the inverse Fourier transform, but gives criterion involving norm estimates on the integral powers of Laplacian acting on the function.   The main papers here are \cite{Bang1990, Bang1995, Tuan-Z, Andersen2004-1, Andersen2004-2, Andersen-DeJu2010, Bang-Huy2010}. While most of these papers deal with Euclidean spaces,  \cite{Andersen2004-1} considers the Riemannian symmetric space, where estimates on the $L^2$-norm of positive integral powers of Laplacian is used. Part (b) of Theorem \ref{main-result} is an extension of this as it characterizes functions whose Fourier transform is supported in a compact annulus around origin, under a weaker norm-condition. A different set of papers  started with  Roe \cite{Roe}  and followed  by many, including \cite{Gabardo1989, Howd-1, Howd-Reese, Str-Roe, KRS2014, RS2014}  try to characterize eigenfunctions of differential operators, in particular of the Laplacian,  from a normed-estimate of a double sequence of functions $\{f_k\}$ related by $\Delta f_k=f_{k+1}$ for $\Delta$ of the space in context. Most of these papers deal with  Euclidean spaces. One important exception is \cite{Str-Roe} where  Strichartz  establishes  the failure of the Euclidean result for   hyperbolic spaces, as mentioned above. But through \cite{KRS2014} and \cite{RS2014} the result is restored for all Riemannian symmetric spaces of noncompact type (which includes hyperbolic spaces) and is also generalized to harmonic $NA$ groups.
A careful study reveals that the common thread between these two sets of results is the use of estimates of integral powers of Laplacian applied on the function.  Our aim is to  offer a  version which accommodates both of these aspects.

We note in passing that  `the compactly supported  Fourier transform' binds the real and the usual inverse Paley-Wiener theorem together,  vindicating a relation between the estimates of $\Delta^n f$ and the  regularity of $f$.
Indeed the use of estimates of iterated action of Laplacian or more general  operators on  a function  to retrieve regularity properties of  the function  is classical. We may cite  for example  Nelson,  Kotake and Narasimhan \cite{Nelson, Kot-Nara} and the references  therein.

\noindent{Acknowledgement:} The author is thankful to  Swagato K Ray for numerous useful discussions during this work.

\section{Preliminaries} \label{prelim}
In this section we shall establish notation and collect all ingredients to explain the statement and proof of the main result.
\subsection{Generalities} \label{pri-gen}
For any $p\in [1, \infty)$, let $p'=p/(p-1)$. The letters $\N$, $\Z$, and  $\R$, $\C$  denote respectively the set of natural numbers, ring of integers, field real  and complex numbers.  We denote the nonzero real numbers, nonnegative real numbers and
nonnegative integers respectively by $\R^\times$, $\R^+$ and $\Z^+$. For
$z\in \C$,  $\Re z$, $\Im z$ and $\bar{z}$  denote respectively the real and
imaginary parts of $z$ and the complex conjugate of $z$.  For  a set $S$ in a topological space $\overline{S}$ denotes its closure and for a set $S$ in a measure space $|S|$ denotes its measure.   We
shall follow the standard practice of using the letters $C, C_1, C_2, C'$
etc. for positive constants, whose value may change from one line to another.
The constants may be suffixed to show their
dependencies on important parameters.
The notation $\langle f_1, f_2\rangle$ for two functions or distributions $f_1, f_2$, is frequently used in this article. It may  mean  $\int f_1 f_2$ when it makes sense. It may also mean  that the distribution $f_1$ is acting on $f_2$. Depending on the functions/distributions $f_1, f_2$ involved, the space could be  $X$ or its Fourier-dual $\R^+\times B$, or $\R$ with the canonical  measures on them. As this notation is widely used in the literature, we hope this will not create any confusion.
For two positive expressions $f_1$ and $f_2$, by $f_1\asymp f_2$ we mean that there are  constants $C_1, C_2>0$ such that
$C_1f_1\leq f_2\leq C_2f_1$.
\subsection{Lorentz spaces} \label{pri-Lorentz} We shall briefly introduce Lorentz
spaces  (see \cite{Graf, S-W, RS2009} for details).
Let $(M, m)$ be a $\sigma$-finite measure space, $f:M\longrightarrow \C$ be a
measurable function and $p\in [1, \infty)$, $q\in [1, \infty]$. We
define
\begin{equation*}\|f\|^*_{p,q}=\begin{cases}\left(\frac qp\int_0^\infty [f^*(t)t^{1/p}]^q\frac{dt}t\right)^{1/q}
\ \ \ \ \ \ \ \ \  \ \ \ \ \ \textup{ if } q<\infty,\\ \\ \sup_{t>0}td_f(t)^{1/p}=\sup_{t>0} t^{1/p} f^\ast(t)
 \ \ \textup{ if } q=\infty,\end{cases}\end{equation*} where
for $\alpha>0$, $d_f(\alpha)=|\{x \mid f(x)>\alpha\}|$, the distribution function of $f$ and
$f^*(t)=\inf\{s \mid d_f(s)\le t\}$, the {\em
decreasing rearrangement} of $f$.
Let $L^{p,q}(M)$ be the set of all measurable $f:M\longrightarrow \C$ such that
$\|f\|^*_{p,q}<\infty$.  We note the following.
\begin{enumerate}
\item[(i)] The space $L^{p,\infty}(M)$ is known as the weak $L^p$-space.
\item[(ii)] $L^{p,p}(M)=L^p(M)$ and $\|\cdot\|_{p,p}^\ast=\|\cdot\|_p$.
\item[(iii)] For $1<p, q<\infty$, the dual space of $L^{p,q}(M)$ is $L^{p', q'}(M)$ and the dual of $L^{p,1}(M)$ is $L^{p',\infty}(M)$.
 \item[(iv)] If $q_1\le q_2\le \infty$ then
     $L^{p, q_1}(M)\subset L^{p, q_2}(S)$ and $\|f\|^\ast_{p, q_2}\le \|f\|^\ast_{p, q_1}$.
\end{enumerate}
The Lorentz  ``norm'' $\|\,\cdot\,\|^\ast_{p,q}$ is actually a quasi-norm  and  $L^{p,q}(M)$ is a quasi Banach
space (see \cite[p. 50]{Graf}). For $1<p\le \infty$, there
is an equivalent norm $\|\,\cdot\,\|_{p,q}$ which makes it a Banach
space (see \cite[Theorems 3.21, 3.22]{S-W}). We shall slur over this difference and use the notation $\|\cdot\|_{p,q}$.

\subsection{Symmetric space} We shall mostly use standard notation for objects related to semisimple Lie groups and the
associated Riemannian symmetric spaces of noncompact type.
Along with required preliminaries this can be found for example in \cite{GV, Helga-2}. For making the article self-contained, we shall  gather  them  without elaboration.
We recall that a rank one Riemannian symmetric space  of noncompact type (which we denote by $X$ throughout  this article) can be realized as a quotient space $G/K$, where  $G$ is a connected noncompact semisimple Lie group with finite centre and of real rank one and $K$  a maximal compact subgroup of $G$. Thus $\pmb{o}=\{K\}$ is the origin of $X$ and a function on $X$ can be identified with a function on $G$ which is invariant under the right $K$-action.
 The group $G$  acts naturally on $X=G/K$ by left translations  $\ell_g: xK\to g^{-1}xK$ for $g\in G$. The Killing form on the Lie algebra $\mathfrak g$ of $G$ induces a $G$-invariant Riemannian structure and a $G$-invariant measure  on $X$. Let $\Delta$ be the corresponding Laplace-Beltrami operator. For an element $x\in X$, let $|x|= d(x, \pmb{o})$, where $d$ is the distance associated to the Riemannian structure on $X$. Let $\mathfrak k$ be  the Lie algebra of $K$, $\mathfrak g=\mathfrak k+\mathfrak p$ be the corresponding Cartan decomposition and  $\mathfrak a$ be a maximal abelian subspace of $\mathfrak p$. Then $\dim\mathfrak a=1$ as  $G$ is of real rank one. We denote the real dual of $\mathfrak a$ by  $\mathfrak a^*$. Let $\Sigma\subset \mathfrak a^*$ be the subset of nonzero roots of the pair $(\mathfrak g,\mathfrak a)$. We recall that either $\Sigma=\{-\gamma, \gamma\}$ or $\{-2\gamma, -\gamma, \gamma, 2\gamma\}$ where $\gamma$ is a positive root and the Weyl group $W$ associated  to $\Sigma$ is $\{{\rm Id}, -{\rm Id}\}$ where Id is the identity operator.
 Let $m_\gamma=\dim \mathfrak g_\gamma$ and $m_{2\gamma}=\dim \mathfrak g_{2\gamma}$ where $\mathfrak g_\gamma$ and
 $\mathfrak g_{2\gamma}$ are the root spaces corresponding to $\gamma$ and $2\gamma$. Then $\rho=\frac 12(m_\gamma+2m_{2\gamma})\gamma$ denotes the half sum of the positive roots. Let  $H_0$ be the unique element in $\mathfrak a$ such that $\gamma(H_0)=1$ and through this we identify $\mathfrak a$ with $\R$ as $t\mapsto tH_0$. Then $\mathfrak a_+= \{H\in \mathfrak a\mid \gamma(H)>0\}$ is identified with the set of positive real numbers. We  identify $\mathfrak a^*$ and its complexification $\mathfrak a^*_\C$ with $\R$  and  $\C$  by $t\mapsto t\gamma$ respectively  $z\mapsto z\gamma$, $t\in \R$, $z\in \C$. By abuse of notation we will denote $\rho(H_0)=\frac 12(m_\gamma+2m_{2\gamma})$ by $\rho$. Let $\mathfrak n= \mathfrak g_\gamma+\mathfrak g_{2\gamma}$, $N=\exp \mathfrak n$,
  $A=\exp \mathfrak a$, $A^+=\exp \mathfrak  a_+$ and $\overline{A^+}=\exp \overline{\mathfrak  a_+}$. Then  $N$ is a nilpotent Lie group and $A$ is a one dimensional vector subgroup identified   with $\R$.
 Precisely $A$ is parametrized by $a_s=\exp(sH_0)$.  The Lebesgue measure on $\R$ induces a Haar measure on $A$ by $da_s=ds$.  Let $M$ be the  centralizer of $A$ in $K$. The groups $M$ and $A$ normalizes $N$.

The group $G$ has  the Iwasawa decomposition
$G=KAN$ and the polar decomposition $G=K\overline{A^+}K$. Through polar decomposition $X$ is realized as $\overline{A^+}\times B$ where $B=K/M$ is the compact boundary of $X$.
Using the Iwasawa decomposition $G=KAN$,  we write an
element $x\in G$ uniquely as $k(x)\exp H(x)n(x)$ where $k(x)\in K, n(x)\in N$ and $H(x)\in \mathfrak a$. For $x\in X=G/K, b\in B=K/M$  let $A(x, b)=A(gK, kM)= -H(x^{-1}k)$. Let $dg$,  $dk$ and
$dm$ be the Haar measures of $G$,  $K$ and $M$ respectively
with $\int_K\,dk=1$ and $\int_M\,dm=1$. Let $db$ be the normalized measure on $K/M = B$ induced by $dk$ on  $K$. We have the following
integral formulae corresponding to the  Iwasawa decompositions $G=KAN$ and the polar decomposition, which hold for any integrable function:
\begin{equation}
\int_Gf(g)dg=C_1\int_K\int_\R\int_N f(ka_tn)e^{2\rho t}\,dn\,dt\,dk,
\label{Iwasawa-1}
\end{equation}
and
\begin{equation}
\int_Gf(g)dg=C_2\int_K\int_0^\infty\int_K f(k_1a_tk_2) (\sinh
t)^{m_\gamma}(\sinh 2t)^{m_{2\gamma}}\,dk_1\,dt\,dk_2, \label{polar}
\end{equation}
The constants  $C_1, C_2$ depend on the normalization of the Haar measures involved.
Since  $\sinh t\approx te^t/(1+t), t\ge 0$ it follows from  (\ref{polar}) that
\begin{eqnarray}
\int_G|f(g)|dg&\asymp&C_3\int_K\int_0^1\int_K |f(k_1a_tk_2)|
t^{d-1}\,dk_1\,dt \,dk_2\nonumber \\&+&C_4 \int_K\int_1^\infty\int_K
|f(k_1a_tk_2)|e^{2\rho t}\,dk_1\,dt\,dk_2 \label{polar-2}
\end{eqnarray} where $d=m_\gamma+m_{2\gamma}+1$ is the dimension of the symmetric space. For a integrable function $f$ on $X$, $\int_Gf(g)dg=\int_Xf(x)dx$ where in the left hand side $f$ is considered as a right $K$-invariant function on $G$ and $dg$ is the Haar measure on $G$, while on the right side $dx$ is the  $G$-invariant measure on $X$.
\subsubsection{Poisson transform} \label{pri-Poisson}
For $\lambda\in \C$, the complex power of the
Poisson kernel: $x\mapsto  e^{-(i\lambda+\rho) H(x^{-1})}$ is an
eigenfunction of the Laplace Beltrami operator $\Delta$ with
eigenvalue $-(\lambda^2+\rho^2)$. For any $\lambda\in \C$ and $F\in
L^1(B)$ we define the  Poisson transform $\mathcal P_\lambda$ of $F$
by (see \cite[p. 279]{Helga-2}) by
\[\mathcal P_\lambda  F(x)=\int_{B} F(b)e^{(i\lambda+\rho) A( x, b)} db\, \, \,  \text{ for  } x\in X.\] Then,
\[\Delta \mathcal P_\lambda F=-(\lambda^2+\rho^2)\mathcal P_\lambda F.\]

A function $f$ on $X$ is left $K$-invariant or {\em radial} if $f(kx)=f(x)$ for all $k\in K$ and $x\in X$. Note that a left $K$-invariant function on $X$ can be identified with a $K$-biinvariant function on $G$. We shall use both the terms radial and $K$-biinvariant for such functions.  For any function space $\mathcal L(X)$, by $\mathcal L(G//K)$ we mean its subset of $K$-biinvariant functions.  For a suitable function $f$ on $X$ we define its  {\em radialization}  $Rf$ by   $Rf(x)=\int_Kf(kx)dk$. It is clear that $Rf$ is a radial function and if $f$ is radial then $Rf=f$. We also note that for (i) $\phi, \psi\in C_c^\infty(X)$, $\langle R\phi, \psi\rangle=\langle \phi, R\psi\rangle$ and (ii) $R(\Delta \phi)=\Delta(R\phi)$. From (i) it follows that $\int_{X} f(x)dx=\int_{X} Rf(x) dx$ and hence $\|Rf\|_1\le \|f\|_1$. Interpolating \cite[p. 197]{S-W}  with the trivial $L^\infty$-boundedness of the operator $R$ we get
\[\|Rf\|_{p,q}\le \|f\|_{p,q}  \text{ for } 1<p<\infty, 1\le q\le \infty.\]

For any $\lambda\in \C$  the elementary spherical function
$\phi_\lambda$ is given by,
\[\phi_\lambda(x)=\mathcal P_\lambda 1(x)=\int_{K} e^{-(i\lambda+\rho)H(xk)}\,dk=\int_{K} e^{(i\lambda-\rho)H(xk)}\,dk \text{ for all } x\in G,\] where by $1$ we denote the constant function $1$ on $B=K/M$. Hence $\Delta \phi_{\lambda} = -(\lambda^2+\rho^2) \phi_{\lambda}$  for $\lambda\in \C$. It follows that for $\lambda\in \C$, $\phi_\lambda$ is radial,
$\phi_{\lambda} = \phi_{-\lambda}$ and it satisfies the following estimates: (see \cite{ADY}, \cite[(4.6.5)]{GV})
\begin{eqnarray}
\label{estimates-phi-lambda-classical-1}
&|\phi_{\alpha+i\gamma_p\rho}(x)|\asymp e^{-(2\rho/p')  |x|}, \  \alpha\in \R, 0<p< 2, \gamma_p=2/p-1; \nonumber\\
 &|\phi_0(a_t)|\le C  e^{\rho t} (1+|t|), \text{ for } t>0 \label{exact-estimate-1}\end{eqnarray}  and
\begin{equation}
\label{estimates-phi-lambda-classical-2}
\left|\frac{d^n}{d\lambda^n} \phi_\lambda(x)\right|\le C (1+ |x|)^n \phi_{\Im \lambda}(x) \text{ for } \lambda\in \C. \end{equation}

\subsubsection{Spherical Fourier Transform} \label{pri-spherical}
For a measurable function $f$ of $X$, we define its {\em{spherical Fourier transform}} $\widehat{f}$ and its inverse as follows (see \cite[p. 425, p. 454]{Helga-2}),  \[ \widehat{f}(\lambda) = \int_{X} f(x) \phi_{-\lambda}(x) \, dx, \quad \lambda \in \mathfrak a^{\ast},\,\,\,  f(x) = C \int_{\mathfrak a^{\ast}} \widehat{f}(\lambda) \, \phi_{\lambda}(x) \, |\hc(\lambda)|^{-2} d\lambda,\]
whenever the integrals make sense. Here  $\hc(\lambda)$ is the Harish-Chandra $\hc$-function, $d\lambda$ is the Lebesgue measure on $\mathfrak a^\ast\equiv \R$ and  $|\hc(\lambda)|^{-2} d\lambda$ is the spherical Plancherel measure on $\mathfrak a^{\ast}$ and $C$ is a normalizing constant.
Since $\phi_\lambda=\phi_{-\lambda}$ we have $\what{f}(\lambda)=\what{f}(-\lambda)$, hence  we can consider $\what{f}$ as a function on $\R^+$.

\subsubsection{Helgason Fourier Transform}\label{pri-HFT}
 For a function  $f$ on $X$,  its Helgason Fourier transform (or Fourier transform) is defined by \[\wtilde{f}(\xi, b)=\int_{X} f(x) e^{(-i\xi+\rho)(A(x, b))} dx\] for  $\xi\in \overline{\mathfrak a^\ast_+} \equiv \R^+$,  $b\in B$ for which the integral exists. (See \cite[pp. 199-203]{Helga-3} for details.) The Fourier transform $f(x)\to \wtilde{f}(\xi, b)$ extends to an isometry of $L^2(X)$ onto $L^2(\mathfrak \R^+\times B, |\hc(\xi)|^{-2} d\xi db)$  and we have,
 \[\int_{X} f_1(x) \overline{f_2(x)} dx=C \int_{\R^+\times B} \wtilde{f_1}(\xi, b)\overline{\wtilde{f_2}(\xi, b)} |\hc(\xi)|^{-2} d\xi db.\]
For functions $f, g$ on $X$ with   $g$ radial,  $\wtilde{g}(\xi, k)=\what{g}(\xi)$  and $\wtilde{f\ast g}(\xi, b)=\wtilde{f}(\xi, b)\what{g}(\xi)$ for $\xi\in \C$ and $b\in B$ whenever the quantities $f\ast g, \wtilde{f\ast g}, \wtilde{f}$ and $\what{g}$ make sense.

\subsubsection{Schwartz spaces, tempered distributions}\label{pri-schw} For $1\le p\le 2$, the $L^p$-Schwartz space $C^p(X)$ is defined (see
\cite{Ank-Sch}) as the set of $C^\infty$-functions on $X$ such that
\[\gamma_{r,D}(f)=\sup_{x\in S}|Df(x)|
\phi_0^{-2/p}(1+|x|)^r<\infty,\] for all nonnegative integers $r$
and left invariant differential operators $D$ on $X$. We topologize $C^p(X)$  by the seminorms $\gamma_{r, D}$.
Then $C^p(X)$ is  a dense subset of $L^p(X)$. Let
$C^p(G//K)$ be the set of radial functions in $C^p(X)$.  We shall primary use $C^2(X)$, the $L^2$-Schwartz space.  Let $C^2(\what{X})$ (respectively $C^2(\what{G//K})$)  be the image of $C^2(X)$  (respectively of $C^2(G//K)$) under $f\mapsto \wtilde{f}$ (respectively $f\mapsto \what{f}$). Then (see
\cite{{Ank-Sch}})   $f\mapsto \what{f}$ is a topological
isomorphism from $C^2(G//K)$ to $C^2(\what{G//K})=S(\R)_{even}$ where $S(\R)$ is  the set of Schwartz class functions on $\R$, and $S(\R)_{even}$ denotes the subspace of even functions in $S(\R)$. We do not need the explicit description of $C^2(\what{X})$, for which along with the isomorphism of $f\mapsto \wtilde{f}$ from $C^2(X)$ to $C^2(\what{X})$ we refer to \cite[Theorem 4.8.1]{egu79}.

We denote the dual space of $C^p(G//K)$ (respectively  $C^p(X)$) by $C^p(G//K)'$ (respectively  $C^p(X)'$).
 Elements of  $C^p(G//K)'$ and  $C^p(X)'$  are called respectively the $K$-bi-invariant $L^p$-tempered distributions and $L^p$-tempered distributions on $X$.
It is clear that $L^{p'}(G//K)\subset C^p(G//K)'$ and $L^{p'}(X)\subset C^p(X)'$ for $1\le p\le 2$.
For an $L^2$-tempered distribution $f$, $\wtilde{f}$ is defined as a continuous linear functional on $C^2(\what{X})$: for $\phi\in C^2(X)$, $\langle \wtilde{f}, \wtilde{\phi}\rangle=\langle f, \phi\rangle$.

For a function   $\phi\in C^2(X)$, we define support of  $\wtilde{\phi}$ as a subset of $\R^+\times  B$ by
 \[\mathrm{Suppt}\, \wtilde{\phi}=\overline{\{(\lambda, b) \in \R^+\times B  \mid \wtilde{\phi}(\lambda, b)\neq 0\}},\]
  If $\phi$  is also  $K$-biinvariant then
 $\wtilde{\phi}(\lambda, b)=\what{\phi}(\lambda)$ for all $b\in B$  and hence  $\mathrm{Suppt }\, \what{\phi}=\overline{\{\lambda \in \R^+  \mid \what{\phi}(\lambda)\neq 0\}}\times B$. When $\phi$ is $K$-biinvariant, by abuse of terminology,   the set $\overline{\{\lambda \in \R^+  \mid \what{\phi}(\lambda)\neq 0\}}$ will also be called support of $\phi$.  We recall   that  $L^{2, \infty}(X)\subset C^2(X)'$ (see Proposition \ref{summary-prop} (ii) below).
 For a function $f\in L^{2, \infty}(X)$,   the distributional support of $\wtilde{f}$ is  the complement of the largest open set $U\subset \R^+\times B$  such that for any $\phi\in C^2(X)$ with $\mathrm{Suppt}\, \wtilde{\phi}$ contained in $U$, $\langle f, \phi\rangle=0$.

If for  a function $f\in L^{2, \infty}(X)$,   $\mathrm{Suppt }\, \wtilde{f}$ is an empty set then $f\equiv 0$. Indeed, $\mathrm{Suppt }\,\wtilde{f}$ is empty implies that $f$ annihilates all functions in $C^{2}(X)$ and hence it is zero as a $L^2$-tempered distribution.

\subsubsection{Abel transform} \label{pri-Abel}
For a  radial  function $f$ on $X$ its Abel transform $\Ab f$ is defined by:
\[\Ab f(a)=e^{\rho(\log a)} \int_N f(an)dn,\text { for } a\in A,\] whenever the integral makes sense.   Through the identification of $A$ with $\R$ we can write it as:
\[\Ab f(t) = e^{\rho t} \int_N f(a_t n) dn \text{ for } t\in \R.\]   For $f\in S(\R)$ let $\mathcal F(f)(\xi)=\int_\R f(x) e^{-i\xi x} dx$ be its Euclidean Fourier transform at $\xi\in \R$.

We recall: (see \cite{Ank-Sch}) (a)  ({\em slice projection} theorem)  for any $f\in C^2(G//K)$, $\lambda\in  \R$,  $\mathcal F(\Ab f)(\lambda)=\what{f}(\lambda)$, (b)   $\Ab: C^2(G//K)\to S(\R)_{even}$ is a topological isomorphism. By duality from  the second statement we get that the adjoint of the Abel transform  $\Ab^\ast:
S(\R)'_{even} \to C^2(G//K)'$  is an  isomorphism (see \cite[p. 541]{Helga-Abel}).

\section{Some preparatory discussions} \label{preparation}
In this section we shall explain the statement of the main result, highlight some of its features and gather some results which will be used in the next section.

(1) As mentioned in the introduction,  the weak $L^2$-norm  in the hypothesis is the only possible Lorentz norm for the formulation. We shall elaborate on this.

As the statement of Theorem \ref{main-result} involves Fourier transform, tempered distribution  is a natural choice to work with. An $L^{2, \infty}$-function on $X$ is an $L^2$-tempered distribution and the space $L^{2, \infty}(X)$ is {\em close to} $L^2(X)$, where usually the inverse Paley-Wiener theorems are stated. We recall that for  $1\le q<\infty$,    $L^{2, q}(X)\subset C^2(X)'$ (see Proposition \ref{summary-prop} (ii) below), i.e. an $L^{2, q}$-function is also an $L^2$-tempered distribution.
But $L^{2, q}$-norms (which in particular includes $L^2=L^{2,2}$-norm)  discards the possibility of $f$ being an eigenfunction (see Proposition \ref{summary-prop} (vi) below). Hence in this case  $c_1c_2>1$.

 Suppose that we take $f\in L^{p, q}(X)$ with $1\le p<2, 1\le q\le\infty$ and use $L^{p, q}$-norm in the hypothesis instead of $L^{2, \infty}$-norm.  Then again $f$ is an $L^2$-tempered distribution. Indeed $C^2(X)\subset L^2(X)\cap L^\infty(X)$ and hence $C^2(X)\subset L^{p', q}(X)$ for $p, q$ in the range above  by interpolation. This implies  by duality that $L^{p, q'}(X)\subset C^2(X)'$.  But Fourier transform $\wtilde{f}(\lambda, b)$ of such a function $f$ which exists point-wise, has complex-analytic extension in $\lambda$ in a strip for almost every $b\in B$ (see \cite{MRSS, RS2009}) and so  if the limits in the hypothesis exist, the only possibilities are $c_1=\infty$ and $c_2=\rho^{-2}$, i.e.  the annulus $\mathbb A^\alpha_\beta=\C\times B$.

 Lastly if $f\in L^{p,q}(X)$ with  $p>2$, $1\le q\le \infty$, then $f$ is an $L^{p'}$-tempered distribution where $p'<2$ (and in general not an $L^2$-tempered distribution).
 See \cite[section 6]{KRS2014}. It is clear that the usual definition of distributional support of its Fourier transform  is  not meaningful for such a function since there is no function in $C^{p'}(X)$ whose Fourier transform is compactly supported.    On the other hand there are  functions  $f\in L^{p, q}(X)$  satisfying
    \[\lim_{n\to \infty} \|\Delta^n f\|_{p, q}^{1/n}=c_1,\,\, \lim_{n\to \infty} \|\Delta^{-n} f\|_{p, q}^{1/n}=1/c_1\] which are not eigenfunctions (not even generalized eigenfunctions) of $\Delta$. An easy example is the following. We take two points $\lambda_1, \lambda_2\in \C$ such that $|\Im \lambda_i|< |2/p-1|  \rho$ and $|\lambda_1^2+\rho^2|=|\lambda_2^2+\rho^2|=\delta$ for some fixed $\delta>(4\rho^2)/(pp')$. Indeed  uncountably many $\lambda\in \C$ satisfy this for any such fixed $\delta$.
    Then it is easy to verify that if $f=\phi_{\lambda_1}+\phi_{\lambda_2}$ then $f$ is not a generalized  eigenfunction but satisfies the hypothesis of Theorem \ref{main-result} with the substitution of $L^{2, \infty}$-norm by $L^{p,q}$-norm for $p, q$ as above.

(2)
 Outside the set of  Lorentz norms and $L^p$-norms there are some prominent size estimates which are used in the literature to characterize eigenfunctions of Laplacian as Poisson transforms. We shall mention only two of them. Let $B(0, r)=\{x\in X \mid |x|<r\}$ be the geodesic ball of radius $r$. For $1<p<\infty$, $1\le q< \infty$ and a function $f$ on $X$ we define
    \begin{eqnarray}
M_p(f)&=&\left(\limsup_{r\to\infty}\frac{1}{r}\int_{B(0,r)}|f(x)|^pdx\right)^{1/p},\\
\mathcal K_{p,q}(f)&=&\|\mathcal K_q(f)\|_{p, \infty}, \text{ where } \mathcal K_q(f)(x)=\left(\int_{K}|f(kx)|^qdk\right)^{1/q}.\label{npq-DR}
\end{eqnarray} Any function $f$ on $X$ satisfying $M_2(f)<\infty$ or $\mathcal K_{2,q}(f)<\infty$ is an $L^2$-tempered distribution. (See the line above Section 4).  Since  the argument in the proof of Theorem \ref{main-result}  works under the assumption  that $f$ is an $L^2$-tempered distribution,  we can substitute $L^{2, \infty}$-norm by $M_2$-norm or by $\mathcal K_{2,q}$-norm.
  See   \cite{KRS2014} for the background relevant to these norms.
\vspace{.15in}

(3) Negative powers of $\Delta$ used in the statement of Theorem \ref{main-result} can be interpreted in terms of radial multipliers. Precisely, $\Delta^{-1}$ is  an $L^p$-multiplier for $1<p<2$ (see \cite{Ank-mult}) and hence  an $L^{p'}$-multiplier. Hence by interpolation \cite[p. 197]{S-W} defines a bounded operator from $L^{2, \infty}(X)$ to itself. This is a benefit of the fact that in $X$ (and  $NA$ groups) the spectrum of $\Delta$ does not contain  $0$  (see \cite{Tay}). But keeping in mind the  spaces (e.g. $\R^n$) where  this interpretation is not valid, we can have an  alternative formulation following \cite{Gabardo1989, Str-Roe, Howd-Reese}, which in our case is only a change of notation.
\begin{theorem} \label{main-result-2}
Let $\{f_k\}_{k\in \Z}$ be a doubly infinite sequence of nonzero functions in $L^{2, \infty}(X)$ with  $\Delta f_k=f_{k+1}$ for all $k\in \Z$.  Suppose for constants $c_1 \ge \rho^2, c_2\le 1/\rho^2$,
\[\lim_{k\to \infty} \|f_k\|_{2, \infty}^{1/k}=c_1,\,\, \lim_{k\to \infty} \|f_{-k}\|_{2, \infty}^{1/k}=c_2.\]
 Then we have the conclusions of Theorem \ref{main-result} for $f=f_0$.
\end{theorem}
Indeed the substitution  $f=f_0$ and $f_k=\Delta^k f_0=\Delta^k f$ for $k\in \Z$ reduces the hypothesis of this theorem to that of  Theorem \ref{main-result}.

(4)
We recall that $\Delta^n$ for $n\in \N$ commutes with translations, precisely $\Delta^n \ell_x f= \ell_x \Delta^n f$ for any $x\in G$ and a locally integrable  function $f$ on $X$. It is also not difficult to see that
$\Delta^{-n} \ell_x f= \ell_x \Delta^{-n} f$ for any $n\in \N$. Similarly it can be verified that $\Delta^n$ for $n\in \Z$ commutes with the radialization operator $R$, i.e. $\Delta^n (R(f))=R(\Delta^n f)$.

(5) We conclude this section  collecting a few  not-so-well-known results, some of which are used in the discussion above and some will be  required  for the main argument.
\begin{proposition}\label{summary-prop}
\begin{enumerate}[{\em (i)}]
\item $C^2(X)$ is a dense subset of $L^{2,1}(X)$ and  there exists a seminorm $\nu$ of $C^2(X)$ such that for all $\phi\in C^2(X)$,
$\|\phi\|_{2,1}\le C\nu(\phi)$.
\item For  $f\in L^{2, \infty}(X)$, there exists a seminorm $\nu$ of $C^2(X)$ such that for all $\phi\in C^2(X)$,
$|\langle f, \phi\rangle|\le C\|f\|_{2, \infty} \nu(\phi)$. That is $f\in L^{2, \infty}(X)$ is an  $L^2$-tempered distribution. Since for any $q<\infty$, $L^{2, q}(X)\subset L^{2, \infty}(X)$ and $\|f\|_{2, \infty}\le \|f\|_{2, q}$,  any $f\in L^{2, q}(X)$ is also an  $L^2$-tempered distribution.
\item Let $1\le q\le \infty$ be fixed. If for a nonnegative radial measure $\mu$ on $X$, $\what{\mu}(0)<\infty$, then   $T_\mu: f\to f\ast \mu$ defines a bounded operator from $L^{2, q}$ to itself and the operator norm satisfies $\|T_\mu\|_{L^{2, q}\to L^{2, q}}\le \what{\mu}(0)$.
\item For $f\in L^{2, \infty}(X)$ and $\psi\in C^2(G//K)$, $\|f\ast \psi\|_{2, \infty}\le \|f\|_{2, \infty} \nu(\psi)$ for some seminorm $\nu$ of $C^2(X)$.
    \item If  a   nonzero function $f$ on $X$ satisfies $\Delta f=-\rho^2 f$, then $f\not\in L^{2, \infty}(X)$. In particular $\phi_0\not\in L^{2, \infty}(X)$.
    \item If  a   nonzero function $f$ on $X$ satisfies $\Delta f=-(\lambda^2+\rho^2) f$, for some $\lambda\in \R^\times$, then $f\not\in L^{2, q}(X)$ for any $q<\infty$.
    \item For any  $\lambda\in \R^\times$, $\phi_\lambda\in L^{2, \infty}(X)$.
    \item Suppose that  a  function $f$ on $X$ satisfies $\Delta f=-(\lambda^2+\rho^2) f$ with  $\lambda \in \R^\times$. Then   $f=\mathcal P_\lambda u$ for some  $u\in L^2(B)$ if and only if $f\in L^{2, \infty}(X)$ and in that case $\|\mathcal P_\lambda u\|_{2, \infty}\le C_\lambda \|u\|_{L^2(B)}$.
   \end{enumerate}
   \end{proposition}
   \begin{proof}
 (i) follows from the definition of $C^2(X)$ and the fact that for an appropriately large $M$, the function $\phi_0(x) (1+|x|)^{-M}\in L^{2,1}(X)$. See \cite[Lemma 6.1.1]{KRS2014}. Denseness of $C^2(X)$ is  a consequence of denseness of $C_c^\infty(X)$ in $L^{2, 1}(X)$.
(ii) is immediate  from (i) and H\"{o}lder's inequality. See also \cite[Lemma 6.1.1]{KRS2014}.
(iii) is a particular case  of a more general result   proved in  \cite[Lemma 3.2.1]{Sarkar-Chaos1} and \cite{ADY}.
 For (iv) we have \begin{eqnarray*}\what{|\psi|}(0)=\int_{X} |\psi(x)|\phi_0(x) dx \le \sup_{x\in X} [|\psi(x)| \phi_0^{-1}(x) (1+|x|)^M]\int_{X} \phi_0^2(x) (1+|x|)^{-M} dx.
 \end{eqnarray*} It follows from the estimate of $\phi_0$ and the measure on $X$ (see Section 2) that $C=\int_{X} \phi_0^2(x) (1+|x|)^{-M} dx<\infty$ for suitably large $M$. We define
  \[\nu(\psi)=\sup_{x\in X} [|\psi(x)| \phi_0^{-1}(x) C(1+|x|)^M]\] to get  $\what{|\psi|}(0)\le \nu(\psi)$.
Thus by (iii), \[\|f\ast\psi\|_{2, \infty}\le \|\,|f|\ast|\psi|\, \|_{2, \infty}=\|T_{|\psi|} (|f|)\|_{2, \infty}\le \|f\|_{2, \infty}\what{|\psi|}(0)=\|f\|_{2, \infty}\nu(\psi).\]

For (v), (vi),  (vii) and (viii) we refer to \cite[Proposition 3.1.1, (2.2.6) and Theorem 4.3.5]{KRS2014} and \cite{Kumar2014}. ((vii) is also a particular case of (viii).)
\end{proof}
For the corresponding results in particular that of  (i), (ii) and (viii) above for $M_2$ norm and $\mathcal K_{2, q}$ norm,  we refer to \cite[Lemma 6.1.1]{KRS2014} and \cite{Bou-Sami, Ion-Pois-1}.

\section{Proof of the main result} \label{section-proof}
This section is devoted to the proof of  Theorem \ref{main-result}. We begin with a few observations and results which relate the support of the Fourier transform of a function on $X$ with the support of the Fourier transform  of its translation and radialization.
\begin{proposition}\label{support-radialization}  Let  $g\in C^2(X)$ and $\lambda\in \R^+$.  Then  $(\lambda, b)\in \mathrm{Suppt }\,\, \wtilde{g}$ for some $b\in B$ if and only if  $\lambda\in \mathrm{Suppt }\,\, \what{R(\ell_x g)}$ for some $x\in G$.
\end{proposition}
\begin{proof}
Note that for $\lambda\in \R$ (see \cite[p. 200]{Helga-3}),
\[\what{R(\ell_x g)}(\lambda)=\what{\ell_x g}(\lambda)=g\ast \phi_\lambda(x^{-1})=\int_{B} \wtilde{g}(\lambda, b)\, e^{(i\lambda+\rho)(A(x^{-1}, b))}\, db=\mathcal P_\lambda\, \wtilde{g}(\lambda, \cdot)(x^{-1}),\]
where in the last equality above we have considered $\wtilde{g}(\lambda, \cdot)$ as a function on $B$.
If $(\lambda, b)\not\in \mathrm{ support }\,\, \wtilde{g}$ for all $b\in B$ then clearly $\lambda\not\in \mathrm{ support }\,\, \what{R(\ell_x g)}$ for all $x\in G$. Conversely, if $\lambda\not\in \mathrm{ support }\,\, \what{R(\ell_x g)}$ for all $x\in G$, then $\mathcal P_\lambda \wtilde{g}(\lambda, \cdot)\equiv 0$. Using simplicity criterion (\cite[pp. 152, 165]{Helga-3}) this implies that $\wtilde{g}(\lambda, \cdot)\equiv 0$.
\end{proof}

\begin{proposition}
\label{HFT-translation}
Let $g\in C^2(X)$. If support of $\wtilde{g}$ intersects the sphere $\{\gamma\}\times B$ for some $\gamma\ge 0$,  then for any $y\in G$, support of  $\wtilde{\ell_y g}$ also intersects $\{\gamma\}\times B$.
\end{proposition}
\begin{proof}
We have
\[\wtilde{\ell_yg}(\xi, kM)=\int_X g(y^{-1}x) e^{(i\xi-\rho)H(x^{-1}k)} dx.\]
With the substitution  $y^{-1}x=z$ and using
the identity  $H(z^{-1}y^{-1}k)=H(y^{-1}k)+H(z^{-1} K(y^{-1}k))$ (\cite[p.200]{Helga-3}) we get from above
\begin{eqnarray*}\wtilde{\ell_y g}(\xi, kM)&=[e^{(i\xi-\rho)H(y^{-1}k)}]\ \ \int_X g(z) e^{(i\xi-\rho)H(z^{-1}K(y^{-1}k))} dz\\
&=[e^{(i\xi-\rho)H(y^{-1}k)}]\ \ \wtilde{g}(\xi, K(y^{-1}k)).\end{eqnarray*}
Suppose that $\wtilde{g}(\gamma, b)\neq 0$ for $b=k_1M$. Let $K(yk_1)=k$. Then  $K(y^{-1}k)=k_1$ and hence  $\wtilde{\ell_y g}(\gamma, kM)\neq 0$, which proves the assertion.
\end{proof}

We note that for Theorem \ref{main-result}, it is  required  to find  only  the inner and outer radii of the support of $\wtilde{f}$. Precisely, outer and inner  radii of support of $\wtilde{f}$ are $\alpha$ and $\beta$ respectively if support of $\wtilde{f}$ is contained in the annulus $[\beta, \alpha]\times B$ but not contained in $[\beta', \alpha']\times B$ when $\beta<\beta'$ or  $\alpha'<\alpha$.
\begin{observation} \label{observation-1}
Let $f\in L^{2, \infty}(X)$. Then the radii of support of $\wtilde{f}$ are the  same as the radii of support of $\ell_x f$ for any $x\in G$. Suppose that radii of support of $\wtilde{f}$ are $\alpha, \beta$.   We take a function $g\in C^2(G/K)$, such that $\mathrm{Suppt}\,\, \wtilde{g}$ is contained in $\{(\lambda, b)\in \R^+\times B \mid \lambda>\alpha\}$. Then by Proposition \ref{HFT-translation}, $\mathrm{Suppt}\,\, \wtilde{\ell_{x^{-1}} g}$ for any $x\in G$, is also contained in $\{(\lambda, b)\in \R^+\times B \mid \lambda>\alpha\}$. Hence $\langle f, \ell_{x^{-1}}g\rangle=0$. Therefore $\langle \ell_x f, g\rangle=\langle f, \ell_{x^{-1}}g\rangle=0$. Since $f$ is a translation of $\ell_x f$, outer radius of support of $\wtilde{\ell_xf}$ is same with outer radius of support of $\wtilde{f}$. Similarly we can show that inner radius of $\wtilde{f}$ and of $\wtilde{\ell_xf}$ are same.
\end{observation}

\begin{observation} \label{observation-2}
 Let $f\in L^{2, \infty}(X)$. Suppose that $\mathrm{Suppt}\,\, \wtilde{f}\subset \{\alpha\}\times B$. Then $\mathrm{Suppt}\,\, \what{R(\ell_x f)}\subset \{\alpha\}$ for any $x\in G$. Indeed if $R(\ell_x f)=0$ we have nothing to show. So we assume  $R(\ell_x f)\neq 0$. We take a function $g\in C^2(X)$ with $\mathrm{Suppt}\,\, \wtilde{g}\subset \{(\lambda, b)\in \R^+\times B \mid \lambda\neq \alpha\}$. By Proposition \ref{support-radialization}, $\mathrm{Suppt}\,\, \what{R(g)}\subset \{\lambda\in \R^+\mid \lambda\neq \alpha\}\times B$. Threfore by Observation \ref{observation-1}, $\langle \ell_x f, Rg\rangle =0$ and hence $\langle R(\ell_x f), g\rangle=\langle \ell_x f, Rg\rangle =0$.
\end{observation}
This makes us ready to  present the proof of the main result.
 For readers' convenience distinguished parts of the proof of (c) are separated as a series of lemmas, given after the proof of this theorem. Lemma \ref{eigen-for-nonradial} (and its generalization Proposition \ref{gen-eigen-for-nonradial} in the next section) may be of independent interest.
\begin{proof}[Proof of Theorem \ref{main-result}]
We shall prove (b) and (c) and then use them to prove (a) and (d).

(b) We take $\lambda_1, \lambda_2\in \R^+$  such that $\alpha<\lambda_1<\lambda_2$. Let $\phi\in C^2(\what{G//K})$ be supported on $[\lambda_1, \lambda_2]$. We claim  that $\langle \wtilde{f}, \phi\rangle=0$.

Let $\epsilon= \frac 14(\lambda_1^2-\alpha^2)>0$ where $\alpha=\sqrt{c_1-\rho^2}$. From the hypothesis we know that there exists $N\in \N$, such that for all $n\ge N$,
\begin{equation}|\,\,\, \|\Delta^nf\|_{2, \infty}^{1/n}-c_1|<\epsilon \text{ and hence } (c_1-\epsilon)^n < \|\Delta^nf\|_{2, \infty}< (c_1+\epsilon)^n.
\label{estimate-hypo}
\end{equation}
As $\wtilde{\Delta^kf}=(-1)^k(\lambda^2+\rho^2)^k \wtilde{f}$ (where $\lambda$ is a dummy variable),
\begin{eqnarray*}
|\langle\wtilde{f}, \phi\rangle|&=|\langle\wtilde{\Delta^k f}, \frac 1{(\lambda^2+\rho^2)^k}\phi\rangle|\\
&=|\langle \Delta^k f, \psi_k\rangle|\\
&\le \|\Delta^k f\|_{2, \infty}\|\psi_k\|_{2,1}\\
&\le \|\Delta^k f\|_{2, \infty} \nu(\psi_k)\\
&\le \|\Delta^k f\|_{2, \infty} \mu(\what{\psi_k})
\end{eqnarray*}
where $\psi_k\in C^2(G//K)$ is the inverse spherical transform  of $(\lambda^2+\rho^2)^{-k}\phi\in C^2(\what{G//K})$ and $\nu, \mu$ are seminorms of $C^2(X)$ and of $C^2(\what{X})$ respectively.  We have used above H\"{o}lder's inequality, that $\|\psi_k\|_{2,1}\le \nu(\psi_k)$ (Proposition \ref{summary-prop} (i)) and the isomorphism between  $C^2(G//K)$ and $C^2(\what{G//K})$ (see subsection \ref{pri-schw}).

Thus  for $k\ge N$, we have
\begin{equation} \label{point1}
|\langle\wtilde{f}, \phi\rangle| \le (c_1+\epsilon)^k \mu(\what{\psi_k})= \mu \left[\left(\frac{\alpha^2+\rho^2+\epsilon}{\lambda^2+\rho^2}\right)^k \phi\right].
\end{equation}
Recall that $\phi$ is supported on $[\lambda_1, \lambda_2]$. For $\lambda\in [\lambda_1, \lambda_2]$ and the $\epsilon$ chosen above,
\[\lambda^2+\rho^2\ge \lambda_1^2+\rho^2=\alpha^2+\rho^2+4\epsilon>\alpha^2+\rho^2+\epsilon.\]
Hence given any $\delta>0$ we can find $N_1\in \N$ with $N_1\ge N$ such that for $k\ge N_1$,  $\mu[\ldots]<\delta$ in \eqref{point1} and hence $|\langle\wtilde{f}, \phi\rangle| <\delta$.
This establishes  the claim and proves that $f$ annihilates any function $\phi\in C^2(G//K)$ such that $\what{\phi}$ is supported in a compact set of $\R^+$ outside $[0, \alpha]$.

A step by step adaptation of this  argument will show that $f$ also annihilates any function $\psi\in C^2(G//K)$ such that $\what{\psi}$ is supported in a compact set of $\R^+$ outside  $[\beta, \infty)$. We include a sketch.
We take $\xi_1, \xi_2$ with $0<\xi_1<\xi_2<\beta$. Let $\phi\in C^2(\what{G//K})$ be supported on $[\xi_1, \xi_2]$. We need to show that $\langle \wtilde{f}, \phi\rangle=0$.
We take
\begin{equation}\label{epsilon-2}
\epsilon= \frac{\beta^2-\xi_2^2}{4(\xi_2^2+\rho^2)(\beta^2+\rho^2)}>0.
\end{equation}
 It follows from the hypothesis that there exists $N\in \N$, such that for all $n\ge N$,
\begin{equation}|\,\,\, \|\Delta^{-n}f\|_{2, \infty}^{1/n}-c_2|<\epsilon \text{ and hence } (c_2-\epsilon)^n < \|\Delta^{-n}f\|_{2, \infty}< (c_2+\epsilon)^n.
\label{estimate-hypo-2}
\end{equation}
Following steps of the previous part of the proof we get
\begin{eqnarray*}
|\langle\wtilde{f}, \phi\rangle|&=|\langle\wtilde{\Delta^{-k} f}, (\lambda^2+\rho^2)^k\phi\rangle| \le \|\Delta^{-k} f\|_{2, \infty} \mu(\what{\psi_k})
\end{eqnarray*}
where $\psi_k\in C^2(G//K)$ is the inverse image of $(\lambda^2+\rho^2)^{k}\phi\in C^2(\what{G//K})$ and $\mu$ is a  seminorm of $C^2(\what{X})$.
Taking  $k\ge N$, we have \[|\langle\wtilde{f}, \phi\rangle| \le (c_2+\epsilon)^k \mu(\what{\psi_k})
 =\mu \left[\left(\frac 1{\beta^2+\rho^2}+\epsilon\right)^k(\lambda^2+\rho^2)^k \phi\right].\]
Since  $\phi$ is supported on $[\xi_1, \xi_2]$, by \eqref{epsilon-2} we have for $\lambda\in [\xi_1, \xi_2]$,
\[4\epsilon+\frac 1{\beta^2+\rho^2}=\frac 1{\xi_2^2+\rho^2} \le \frac 1{\lambda^2+\rho^2}.\] The rest of the argument is same as the first part.

We have shown that $f$ annihilates any function $\psi\in C^2(G//K)$ with $\what{\psi}$  compactly supported outside $[\beta, \alpha]$. We shall now remove the condition of $K$-biinvariantness from $\phi$.
By Observation \ref{observation-1},
for any $x\in G$, $\ell_xf$ also annihilates all  $\psi\in C^2(G//K)$ for which $\what{\psi}$ is compactly supported outside $[\beta, \alpha]$. Since $\psi(x)=\psi(x^{-1})$, this implies that $f\ast \psi(x)=0$ for all $x\in G$. Noting that $f\ast \psi\in L^{2, \infty}(X)$ (Proposition \ref{summary-prop} (iv)) we have for any $g\in C^2(X)$, $\langle f\ast \psi, g\rangle=0$ and hence by Fubini's theorem $\langle f, g\ast \psi\rangle=0$.

We take $g\in C^2(X)$ with $\mathrm{Suppt }\,\,\wtilde{g}$  contained in an open set $U\subset \R^+\times B$  such that $([\beta, \alpha]\times B)\cap U=\emptyset$. We find another  open set $U_1\subset \R^+\times B$ satisfying  $U\subset U_1$, $U_1$ is $B$-invariant (i.e. if $(\lambda, b)\in U_1$ for some $b\in B$, then $\{\lambda\}\times B\subset U_1$)  and $([\beta, \alpha]\times B)\cap U_1=\emptyset$. We take a $\psi\in C^2(G//K)$ such that $\what{\psi}$ is supported on $U_1$ and $\what{\psi}\equiv 1$ on $U$ (hence on the set $\{\lambda\mid (\lambda, b)\in U \text{ for some } b\in B\}\times B$). Then $g\ast \psi=g$ since $\wtilde{g\ast \psi}(\lambda, k)=\wtilde{g}(\lambda, k) \what{\psi}(\lambda)=\wtilde{g}(\lambda, k)$. Thus by the argument above, $\langle f, g\rangle=\langle f,  g\ast \psi\rangle=0$.

Thus it follows that $\wtilde{f}$ is supported on a subset of  $[\beta, \alpha]\times B$.
We shall now show that it is not supported in a smaller annulus.
We define
\[R^+_f=\sup \{\lambda^2+\rho^2 \mid (\lambda, b) \in \mathrm{Suppt } \wtilde{f}\},\, R^-_f= \inf \{\lambda^2+\rho^2 \mid (\lambda, b) \in \mathrm{Suppt } \wtilde{f}\}.\] Above we have proved that
$c_1\ge R^+_f$ and $1/c_2\le R^-_f$.
Now we shall show that given any $\epsilon>0$, $c_1<R^+f+\epsilon$ and $1/c_2>R^-_f-\epsilon$. For this we fix an $\epsilon>0$.
We take a $\psi\in C^2(G//K)$ such that $\what{\psi}$ is compactly supported,  $\what{\psi}\equiv 1$ on the support of $\wtilde{f}$ (hence $\mathrm{Suppt}\, \wtilde{f}\subseteq \mathrm{Suppt}\, \what{\psi}$) and $R^+_f<R^+_\psi<R^+_f+\epsilon, R^-_f-\epsilon<R^-_\psi<R^-_f$. Then $\what{\psi}\wtilde{f}=\wtilde{f}$ and hence $f=f\ast \psi$.
Thus by Proposition \ref{summary-prop} (iv) and isomorphism of $C^2(G//K)$ and $C^2(\what{G//K})$, there exist  seminorms $\nu$ of $C^2(X)$ and $\mu$ of $C^2(\what{X})$ such that \[\|\Delta^n f\|_{2, \infty}= \|\Delta^n f\ast \psi\|_{2, \infty}=\|f\ast \Delta^n \psi\|_{2, \infty}\le \|f\|_{2, \infty}\, \nu(\Delta^n\psi) \le \|f\|_{2, \infty} \mu(\what{\Delta^n \psi}).\]
Thus,
\begin{eqnarray*}
\|\Delta^n f\|_{2, \infty}\le & \|f\|_{2, \infty} \mu((\lambda^2+\rho^2)^n\what{\psi})\le \|f\|_{2, \infty} (R^+_\psi)^n n! C_{\psi, \mu}
\end{eqnarray*} for some finite constant $C_{\psi, \mu}$ which depends on $\psi$ and $\mu$.
This implies \[c_1=\lim_{n\to \infty} \|\Delta^n f\|^{1/n}\le R^+_\psi< R^+_f+\epsilon.\]

Replacing $\Delta^n f$ by $\Delta^{-n} f$ in the argument above,  we get similarly,
\[\|\Delta^{-n} f\|_{2, \infty}\le \|f\|_{2, \infty} (R^-_\psi)^{-n} n! C_{\psi', \mu}\] which implies
$c_2=\lim_{n\to \infty} \|\Delta^{-n} f\|^{1/n}\le (R^-_\psi)^{-1}$, hence $1/c_2\ge R^-_\psi \ge R^-_f-\epsilon$. This completes the proof of part (b)

(c)  If $c_1c_2=1$ then $\alpha=\beta$, hence $\wtilde{f}$ is supported on the sphere $\{\alpha\}\times B$ of radius $\alpha$. Therefore (c) follows from Lemma \ref{eigen-for-nonradial}.

(a) We have used the two conditions of the hypothesis independently  to prove that $\wtilde{f}$ is supported in a subset of $[0, \alpha]$ and also in a subset of $[\beta, \infty)$ for $\alpha, \beta\in \R^+$. If $\alpha<\beta$ then the support of $\wtilde{f}$ is empty  and hence $f=0$, contradicting the hypothesis. Therefore $\alpha\ge \beta$, equivalently $c_1c_2\ge 1$.

(d) When $\beta=0$ equivalently $c_2=1/\rho^2$ then the annulus $\mathbb A_\beta^\alpha$ obviously reduces to a ball around origin of radius $\alpha$. If $c_1=\rho^2$, then $\alpha=0$ implying that $\alpha=\beta=0$ by (a)   and the interval $[\beta, \alpha]$ degenerates to a singleton set $\{0\}$. Hence by (c) $f$ is an eigenfunction with eigenvalue $-\rho^2$. Then by Proposition \ref{summary-prop} (v), $f\not\in L^{2, \infty}(X)$ contradicting the hypothesis. Therefore $c_1>\rho^2$, equivalently $\alpha>0$, i.e. the $\mathbb A_\beta^\alpha$ does not collapse to origin.
\end{proof}

We shall now prove the lemmas to complete the proof of  (c). We shall write $\partial_\lambda$, $\partial^n_\lambda$ respectively for $\frac{d}{d\lambda}$ and $\frac{d^n}{d\lambda^n}$.
\begin{lemma} \label{polynomial} For any nonconstant polynomial $P$ and  $\lambda_0 >0$, $P(\partial_\lambda) \phi_\lambda|_{\lambda=\lambda_0}\not\in L^{2, \infty}(X)$.
\end{lemma}
\begin{proof}     Let $P$ be a polynomial  of degree $n$ given by $P(y)=a_0y^n+a_1y^{n-1}+\ldots+a_n, a_0\neq 0$. We shall show that if $P(\partial_\lambda) \phi_\lambda|_{\lambda=\lambda_0}\in L^{2, \infty}(X)$, then $\partial_\lambda \phi_\lambda|_{\lambda=\lambda_0}\in L^{2, \infty}(X)$. Use of Lemma \ref{polynomial-1} then completes the proof.

So, we assume that  $P(\partial_\lambda) \phi_\lambda|_{\lambda=\lambda_0}\in L^{2, \infty}(X)$.
If $n=1$, then $P(\partial_\lambda) \phi_\lambda=a_0 \partial_\lambda \phi_\lambda+a_1 \phi_\lambda$. Since $P(\partial_\lambda) \phi_\lambda|_{\lambda=\lambda_0}\in L^{2, \infty}(X)$ and $a_1 \phi_{\lambda_0}\in L^{2, \infty}(X)$ (see Proposition \ref{summary-prop} (vii)) we have $a_0 \partial_\lambda \phi_\lambda|_{\lambda=\lambda_0}\in L^{2, \infty}(X)$. If $n\ge 2$,  we  take a  function $\psi\in C^2(G//K)$ such that $\what{\psi}$ and its derivatives of orders up to $(n-2)$ are  zero at $\lambda_0$. Then $\partial_\lambda^{n-r}(\what{\psi}(\lambda)\phi_\lambda)|_{\lambda=\lambda_0}=0$ for all $2\le r\le n$.

We note that $P(\partial_\lambda) \phi_\lambda|_{\lambda=\lambda_0}\ast \psi=P(\partial_\lambda) (\phi_\lambda\ast \psi)|_{\lambda=\lambda_0}$ where the convolution can be justified from the estimate of $P(\partial_\lambda) \phi_\lambda$ (see \eqref{estimates-phi-lambda-classical-1}, \eqref{estimates-phi-lambda-classical-2}).
Hence,
\[P(\partial_\lambda) \phi_\lambda|_{\lambda=\lambda_0}\ast \psi=\{a_0\partial_\lambda^n (\what{\psi}(\lambda)\phi_{\lambda})+a_1 \partial_\lambda^{n-1} (\what{\psi}(\lambda)\phi_{\lambda})\}|_{\lambda=\lambda_0}.\] Expanding the derivatives in the right hand side by Leibnitz rule and using that   $\what{\psi}$ and its derivatives of order $1, 2, \ldots, n-2$ vanish at $\lambda_0$ we get,
\[P(\partial_\lambda) \phi_\lambda|_{\lambda=\lambda_0}\ast \psi= [a_0\{\phi_\lambda\partial_\lambda^n (\what{\psi}(\lambda)) + n (\partial_\lambda \phi_\lambda) \partial_\lambda^{n-1} (\what{\psi}(\lambda))\}+a_1 \phi_\lambda \partial_\lambda^{n-1} (\what{\psi}(\lambda))]_{\lambda=\lambda_0}.\]

The assumption  $P(\partial_\lambda) \phi_\lambda|_{\lambda=\lambda_0}\in L^{2, \infty}(X)$ implies that  $P(\partial_\lambda) \phi_\lambda|_{\lambda=\lambda_0}\ast \psi\in L^{2, \infty}(X)$ (Proposition \ref{summary-prop}  (iv)). Since $\phi_{\lambda_0}\in L^{2, \infty}(X)$ (Proposition \ref{summary-prop} (vii)), we get from above that $\partial_\lambda \phi_\lambda|_{\lambda=\lambda_0}\in  L^{2, \infty}(X)$.
 \end{proof}
 \begin{lemma}\label{polynomial-1}
 For any $\lambda_0\in \R^+$,  $\partial_\lambda \phi_\lambda|_{\lambda=\lambda_0}\not\in  L^{2, \infty}(X)$.
\end{lemma}
\begin{proof}
In view of the polar decomposition and the corresponding integral formula  \eqref{polar-2}  and  the identification of $A$ with $\R$,  it suffices to show that $\partial_\lambda \phi_\lambda|_{\lambda=\lambda_0}$ restricted to $[1, \infty)$ does not belong to $L^{2, \infty}([1, \infty), e^{2\rho t}dt)$. We shall use the facts that $e^{-\rho t}\in L^{2, \infty}([1, \infty), e^{2\rho t}dt)$ and  $te^{-\rho t}\not \in L^{2, \infty}([1, \infty), e^{2\rho t}dt)$, which are easily verifiable through straightforward computation.  We recall that $\phi_\lambda$ for any $\lambda\in \R^\times$, has the following expansion (see \cite{Stan-Tom, Ion-Pois-1})
\[\phi_\lambda(t)=e^{-\rho t}[\hc(\lambda)e^{i\lambda t}+\hc(-\lambda)e^{-i\lambda t}+E(\lambda, t)],\] where
\[E(\lambda, t)=\hc(\lambda)e^{i\lambda t} \sum_{k=1}^\infty \Gamma_k(\lambda) e^{-2kt} + \hc(-\lambda)e^{-i\lambda t} \sum_{k=1}^\infty \Gamma_k(-\lambda) e^{-2kt}\] and
$\Gamma_k$ are recursively defined by $\Gamma_0(\lambda)=1$ and
\[(k+1)(k+1-i\lambda) \Gamma_{k+1} =(\rho+k) (\rho+k-i\lambda) \Gamma_k+ m_{2\gamma} \sum_{j=0}^k (-1)^{k+j+1} (\rho+2j-i\lambda) \Gamma_j.\]
For $t\ge 1$ the series defining $E(\lambda, t)$ and its $\lambda$-derivative at $\lambda=\lambda_0$ are uniformly convergent. Term by term differentiation shows that $|E(\lambda, t)|\le C_\lambda$ for some constant $C_\lambda$ for $t\ge 1$. Thus $e^{-\rho t} E(\lambda, t)\in L^{2, \infty}([1, \infty), e^{2\rho t}dt)$. Therefore we need to show that
\[e^{-\rho t}\partial_\lambda [\hc(\lambda)e^{i\lambda t}+\hc(-\lambda)e^{-i\lambda t}]|_{\lambda=\lambda_0}\not\in  L^{2, \infty}([1, \infty), e^{2\rho t}dt).\] Noting that $\overline{\hc(\lambda)}=\hc(-\lambda)$  and writing $\hc(\lambda)=a(\lambda)+i\, b(\lambda)$ where $a(\lambda), b(\lambda)$ are real functions, we have
\begin{eqnarray*}e^{-\rho t}\partial_\lambda [\hc(\lambda)e^{i\lambda t}+\hc(-\lambda)e^{-i\lambda t}]=&2e^{-\rho t}\partial_\lambda \,(\Re(\hc(\lambda) e^{i\lambda t}))\\=&2e^{-\rho t}\partial_\lambda (a(\lambda)\cos \lambda t - b(\lambda)\sin \lambda t)\\
=&-2te^{-\rho t} (a(\lambda)\sin \lambda t + b(\lambda)\cos \lambda t)\\
+& 2e^{-\rho t} (\partial_\lambda(a(\lambda))\cos \lambda t - \partial_\lambda(b(\lambda))\sin \lambda t).
\end{eqnarray*}
Since at $\lambda=\lambda_0$  the last term in the equality above is in $L^{2, \infty}([1, \infty), e^{2\rho t}dt)$, we  need only to show that $g(t)=te^{-\rho t} (a(\lambda_0)\sin \lambda_0 t + b(\lambda_0)\cos \lambda_0 t)\not\in L^{2, \infty}([1, \infty), e^{2\rho t}dt)$. For the sake of meeting a contradiction, we assume that $g\in L^{2, \infty}([1, \infty), e^{2\rho t}dt)$. Then its translation by  $\pi/2\lambda_0$ is  $g(t+\pi/2\lambda_0)=C(t+\pi/2\lambda_0)e^{-\rho t} (a(\lambda_0) \cos\lambda_0 t-b(\lambda_0) \sin \lambda_0 t)$, which  is also in $L^{2, \infty}([1, \infty), e^{2\rho t}dt)$. This follows from  interpolation of the facts that for $1<p<2<q$, translation by a fixed element in $\R$ is a bounded operator  from $L^p$ to $L^p$ and from $L^q$ to $L^q$ in the measure space $([1, \infty), e^{2\rho t} dt)$.

We note that the part
$C(\pi/2\lambda_0) e^{-\rho t} (a(\lambda_0) \cos\lambda_0 t-b(\lambda_0) \sin \lambda_0 t)$ of $g(t+\pi/2\lambda_0)$ is  in $L^{2, \infty}([1, \infty), e^{2\rho t}dt)$. Therefore the other part  of $g(t+\pi/2\lambda_0)$, given by
$h(t)=t e^{-\rho t} (-b(\lambda_0) \sin \lambda_0 t +a(\lambda_0) \cos\lambda_0 t) \in L^{2, \infty}([1, \infty), e^{2\rho t}dt)$. Since $g(t)$ and $h(t)$ are in  $L^{2, \infty}([1, \infty), e^{2\rho t}dt)$ we have,
\begin{eqnarray*}
b(\lambda_0) g(t)+a(\lambda_0) h(t)&= te^{-\rho t} (a(\lambda_0)^2+b(\lambda_0)^2) \cos \lambda_0 t\in L^{2, \infty}([1, \infty), e^{2\rho t}dt),\\
a(\lambda_0) g(t) - b(\lambda_0) h(t)&= te^{-\rho t} (a(\lambda_0)^2+b(\lambda_0)^2) \sin \lambda_0 t\in L^{2, \infty}([1, \infty), e^{2\rho t}dt).\\
\end{eqnarray*}
Hence $(a(\lambda_0)^2+b(\lambda_0)^2) e^{i\lambda_0 t} te^{-\rho t}\in L^{2, \infty}([1, \infty), e^{2\rho t}dt)$ which amounts to say that $te^{-\rho t}\in L^{2, \infty}([1, \infty), e^{2\rho t}dt)$, a contradiction.
\end{proof}

\begin{lemma} \label{use-abel} For any polynomial $P$ in one variable and  $\xi\in \R$, $\Ab^\ast(P(\partial_\xi) e^{-i\xi t})=P(\partial_\xi)\phi_\xi$ as $L^2$-tempered distribution on $X$, equivalently  $(\Ab^\ast)^{-1} (P(\partial_\xi) \phi_\xi)=P(\partial_\xi) e^{-i\xi t}$ as tempered distribution on $\R$.
 \end{lemma}
   \begin{proof}
Enough to show this for $P(\partial_\xi)=\partial_\xi$. Let $\psi\in C^2(G//K)$. Then $\Ab \psi\in S(\R)_{even}$.  We have
 \[\langle \Ab \psi, \partial_\xi e^{-i\xi t}\rangle = \langle \psi, \Ab^\ast(\partial_\xi e^{-i\xi t})\rangle.\] On the other hand  using slice-projection theorem (see subsection \ref{pri-Abel}) we have,
 \[\langle \Ab \psi, \partial_\xi e^{-i\xi t}\rangle = \partial_\xi \mathcal F(\Ab \psi)(\xi)= \partial_\xi\what{\psi}(\xi)=\partial_\xi \langle \psi, \phi_\xi\rangle=\langle \psi, \partial_\xi\phi_\xi\rangle.\]
 Thus $\langle \psi, \Ab^\ast(\partial_\xi e^{-i\xi t})\rangle=\langle \psi, \partial_\xi\phi_\xi\rangle$,  for all $\psi\in C^2(G//K)$, implying
 $\Ab^\ast(\partial_\xi e^{-\xi t})=\partial_\xi\phi_\xi$ as $L^2$-tempered distributions.  As $\Ab^\ast$ is an isomorphism from $S(\R)_{even}$ to $C^2(G//K)$, the  equivalent statement follows.
\end{proof}

\begin{lemma}
\label{radial-nonradial}
Let $f_1, f_2$ be two nonzero functions in $L^{2, \infty}(X)$. Then the following statements are true.
\begin{enumerate}[{\em (a)}]
\item There exists $x\in G$ such that $R(\ell_x f_1)\neq 0$.
\item If  for some $x\in G$, $R(\ell_x f_1)\neq 0$, then $R(\Delta^n\ell_x f_1)\neq 0$ for all $n\in \Z$.
\item If $R(\ell_x f_1) = R(\ell_x f_2)$ for all $x\in G$, then $f_1=f_2$.
\end{enumerate}
\end{lemma}
\begin{proof} If $R(\ell_x f_1)=0$  for all $x\in G$, then for any $h\in C^2(G//K)$,
$\langle\ell_x f_1, h\rangle=0$.  Let $h_t, t>0$ be the heat kernel which is an element  in $C^2(G//K)$ defined through its spherical Fourier transform $\what{h_t}(\lambda)=e^{-t(\lambda^2+\rho^2)}$. Taking $h=h_t$ we thus get  $\langle\ell_x f_1, h_t\rangle=0$, i.e.  $f_1\ast h_t\equiv 0$ for all $t>0$. But  $f_1\ast h_t\rightarrow f_1$ as $t\rightarrow 0$ in the sense of distribution.
Therefore $f_1=0$ which contradicts that $f_1$ is nonzero. This proves (a). Applying this on $f_1-f_2$ we get (c).

For (b) it is enough to show that $R(\ell_x f_1)\neq 0$ implies that  $\Delta^{-1}R(\ell_x f_1)\neq 0$ and $\Delta R(\ell_x f_1)\neq
0$.  Indeed $\Delta^{-1}R(\ell_x f_1) = 0$ implies $R(\ell_x f_1)=\Delta\Delta^{-1}R(\ell_x f_1)=0$.
On the other hand if  $\Delta R(\ell_x f_1)= 0$, then $\langle
\Delta R(\ell_x f_1), \psi\rangle=0$ and hence $\langle R(\ell_x f_1), \Delta
\psi\rangle=0$ for all $\psi\in C^2(G//K)$.  Since  for any $\phi\in C^2(G//K)$,
$\what{\phi}(\lambda)(\lambda^2+\rho^2)^{-1}\in C^2(\what{G//K})$ (see subsection \ref{pri-schw}),
$\phi$ can be written as $\phi=\Delta\psi$ for some $\psi\in
C^2(G//K)$. Thus  $\langle R(\ell_x f_1),
\phi\rangle=0$ for any $\phi\in C^2(G//K)$, i.e.  $R(\ell_x f_1)=0$.
\end{proof}
\begin{lemma}\label{eigen-for-radial}
Suppose that the support of the (distributional) spherical Fourier transform $\what{f}$ of $f\in L^{2, \infty}(G//K)$ is $\{\alpha\}$ for some $\alpha>0$. Then
$f=c\phi_\alpha$ for some constant $c$.
\end{lemma}

\begin{proof}
 Since $f$ is an $L^2$-tempered distribution (Proposition \ref{summary-prop} (ii)), $(\Ab^\ast)^{-1} f$ is an even tempered distribution on $\R$ (see subsection \ref{pri-Abel}).  We recall that $C^2(\what{G//K})=S(\R)_{even}$. The Euclidean Fourier transform of $(\Ab^\ast)^{-1}$ in the sense of distribution denoted by
$\mathcal F((\Ab^\ast)^{-1}f)$ is same as the spherical Fourier transform of $f$ in the sense of $L^2$-tempered distribution denoted by $\what{f}$.  Indeed, we take  $\phi, \psi\in S(\R)_{even}$ such that $\mathcal F(\psi)=\phi$. As Abel transform  is an isomorphism between $C^2(G//K)$ and $ S(\R)_{even}$, there is $g\in C^2(G//K)$ such that $\Ab g=\psi$, hence by slice-projection theorem $\what{g}=\mathcal F(\psi)$. Then we have
    \begin{eqnarray*}
     \langle \mathcal F((\Ab^\ast)^{-1}f), \phi\rangle &= \langle (\Ab^\ast)^{-1}f, \psi \rangle
     = \langle (\Ab^\ast)^{-1}f, \Ab g \rangle\\
     &= \langle \Ab^\ast (\Ab^\ast)^{-1}f, g \rangle = \langle f, g \rangle = \langle \what{f}, \what{g} \rangle\\
     &=\langle \what{f}, \mathcal F(\psi) \rangle=\langle \what{f}, \phi \rangle.
    \end{eqnarray*}
    Thus $\langle \mathcal F((\Ab^\ast)^{-1}f), \phi\rangle=\langle \what{f}, \phi \rangle$ where in the left hand side $\phi$ is interpreted as a function of $S(\R)_{even}$ and on the right hand side $\phi$ is an element of $C^2(\what{G//K})$.
Therefore    $\mathcal F((\Ab^\ast)^{-1} f)$ is supported on $\{\alpha\}$.

Therefore by \cite[Theorem 6.25]{Rudin-FA},    \[(\Ab^\ast)^{-1} f(t)=[P_1(\partial_\lambda) e^{i\lambda t}+ P_2(\partial_\lambda) e^{-i\lambda t}]|_{\lambda=\alpha}\] for two polynomials $P_1$ and $P_2$.

 As $\phi_\lambda=\phi_{-\lambda}$ we have by Lemma \ref{use-abel}, $f=P(\partial_\lambda) \phi_\lambda|_{\lambda=\alpha}$ for some polynomial $P$.
Since $f\in L^{2, \infty}(X)$, by Lemma \ref{polynomial} the polynomial is constant. Hence $f=c \phi_\alpha$ for some constant $c$.
\end{proof}
\begin{lemma} \label{eigen-for-nonradial}
Suppose that  for a function $f\in L^{2, \infty}(X)$, $\wtilde{f}$ is supported on the sphere of radius $\alpha>0$ in $\R^+\times B$. Then $f=\mathcal P_\alpha F$ for some $F\in L^2(B)$.
\end{lemma}
\begin{proof} By Observation \ref{observation-2},  for any $x\in G$ either
$R(\ell_x f)$ is  zero or its spherical Fourier transform is supported on $\{\alpha\}$. We also note that since $f \in L^{2, \infty}(X)$, $R(\ell_x f)\in L^{2, \infty}(G//K)$. Therefore by Lemma \ref{eigen-for-radial}, $\Delta R(\ell_x f)=-(\alpha^2+\rho^2) R(\ell_x f)$ for all $x\in G$. That is
 $R(\ell_x \Delta f)= R(\ell_x [-(\alpha^2+\rho^2) f])$ for all $x\in G$. Hence by Lemma \ref{radial-nonradial} (c) $\Delta f = -(\alpha^2+\rho^2) f$. Since $f\in L^{2, \infty}(X)$, by Proposition \ref{summary-prop} (viii), we have $f=\mathcal P_\alpha u$ for some $u \in L^2(B)$.
\end{proof}
\section{Concluding Remarks}
(1) As mentioned earlier, Lemma \ref{eigen-for-nonradial} may be considered as an independent result. We have the following generalization. See \cite[pp. 205]{Howd-Reese}, \cite[Lemma 2.2]{Str-Roe} for Euclidean results of this genre.
\begin{proposition} \label{gen-eigen-for-nonradial}
Suppose that a locally integrable  function $f$ on $X$ satisfies $f(x)(1+|x|)^{-M}\in L^{2, \infty}(X)$ for some fixed nonnegative integer $M$ and $\wtilde{f}$ is supported on the sphere $\{\alpha\}\times B$ of radius $\alpha>0$ in $\R^+\times B$. Then $(\Delta+\alpha^2+\rho^2)^{M+1} f=0$, i.e. $f$ is  a generalized eigenfunction of $\Delta$ with eigenvalue $-(\alpha^2+\rho^2)$. In particular if $M=0$ then $f$ is an eigenfunction.
\end{proposition}
We include  a sketch of the proof.
\begin{proof} We have $\ell_x (f(y)/(1+|y|)^M) =\ell_x f(y)/(1+|x^{-1}y|)\in L^{2, \infty}(X)$. Since $(1+|xy|)<(1+|x|)(1+|y|)$ (\cite[Prop. 4.6.11]{GV}, $\ell_xf(y)/(1+|y|)\in L^{2, \infty}(X)$. Now  as $R(\ell_x f)(y)/(1+|y|)^M)= R(\ell_x f)(y)/(1+|y|)^M)$, we have $R(\ell_x f)(y)/(1+|y|)^M\in L^{2, \infty}(G//K)$. Thus $R(\ell_x f)$ is a $L^2$-tempered distribution.
By Observation \ref{observation-2},  if for some $x\in G$, $R(\ell_x f)\neq 0$ then $\what{R\ell_x f}$ is supported on $\{\alpha\}$.
We fix $x\in G$, such that $R(\ell_x f)\neq 0$. Proceeding as the proof of Lemma \ref{eigen-for-radial} we conclude that $R(\ell_x f)=P_x(\partial_\lambda) \phi_\lambda|_{\lambda=\alpha}$ where the polynomial $P_x$ depends on $x\in G$. Hence by Lemma \ref{who-gen-eigen} proved below, $(\Delta+\alpha^2+\rho^2)^{\deg P_x+1} R(\ell_x f)=0$.  However, the condition $R(\ell_xf)/(1+|\cdot|)^M\in L^{2, \infty}(G//K)$ puts an upper bound for the degree of polynomial, precisely $\deg P_x\le M$ as can be proved going through the steps similar to Lemma \ref{polynomial-1}. Thus for all $x\in G$, $(\Delta+\alpha^2+\rho^2)^{M+1} R(\ell_x f)=0$. That is $R(\ell_x (\Delta+\alpha^2+\rho^2)^{M+1} f)=0$ and hence by Lemma \ref{radial-nonradial}, $(\Delta+\alpha^2+\rho^2)^{M+1} f=0$.
\end{proof}

Now we shall prove the lemma used in the proposition above.
\begin{lemma}\label{who-gen-eigen}
If $e_\lambda$ is an eigenfunction of $\Delta$ with eigenvalue $A(\lambda)$ then for any polynomial $P$ in one variable of degree  $m\in \N$, $(\Delta -A(\lambda))^{m+1} P(\partial_\lambda) e_\lambda =0$ i.e. $P(\partial_\lambda) e_\lambda$ is a generalized eigenfunction of $\Delta$ with eigenvalue $A(\lambda)$.
\end{lemma}
\begin{proof} It suffices to show that $(\Delta -A(\lambda))^{m+1} \partial_\lambda^m e_\lambda =0$, which can be verified by
straightforward computation  for $m=1, 2$. Then we use induction. Suppose the result is true for $m=1, 2, \ldots, n-1$. Now,
  \[(\Delta -A(\lambda))^{n+1} \partial_\lambda^n e_\lambda=(\Delta -A(\lambda))^{n}[\partial_\lambda^n (A(\lambda)e_\lambda) - A(\lambda) \partial_\lambda^n e_\lambda].\] Expanding the part in square bracket $[\ldots]$ in the right hand side above by Leibnitz rule we see that each term in it is of the form $C\partial_\lambda^r A(\lambda)
\partial_\lambda^{n-r} e_\lambda$ for $r=1,2, \ldots, n$. From  induction hypothesis it follows that  $(\Delta -A(\lambda))^{n} \partial_\lambda^{n-r} e_\lambda=0$. This completes the proof.
\end{proof}

Proposition \ref{gen-eigen-for-nonradial} vindicates a generalization of Theorem \ref{main-result}. For a fixed $M>0$ we define a weighted norm $\|\cdot\|_M$ in the following way.
For   measurable function $f$ on $X$, let  $g(x)=f(x)(1+|x|)^{-M}$. Then  $\|f\|_M=\|g\|_{2, \infty}$.
\begin{theorem} \label{towards-distribution-result}
 Let $f$ be a nonzero measurable function on $X$ with $\|f\|_M<\infty$.
Suppose for  constants $c_1\ge \rho^2, c_2\le 1/\rho^2$
\[\lim_{n\to \infty} \|\Delta^n f\|_M^{1/n}=c_1,\,\, \lim_{n\to \infty} \|\Delta^{-n} f\|_M^{1/n}=c_2.\]
Let $\beta=\sqrt{1/c_2-\rho^2} \text{ and } \alpha=\sqrt{c_1-\rho^2}$. Then we have  conclusions {\em (a)} and {\em (b)} of {\em Theorem \ref{main-result}}, while {\em (c)} and {\em (d)} of that theorem are replaced by
\begin{enumerate}
\item[{\em (c)}] If $c_1c_2=1$ then $f$ is  a generalized eigenfunction  with eigenvalue $-c_1$,
\item[{\em (d)}] The annulus $\mathbb A_\beta^\alpha$ may reduce to a ball around origin and may also collapse to the origin.
\end{enumerate}
\end{theorem}
Proof of Theorem \ref{main-result} can be easily adapted to prove this, which we omit for brevity. We only note that under the norm-condition here which is more relaxed than that of Theorem \ref{main-result}, this theorem allows collapsing of the annulus to the origin (see (d) above). This corresponds to the case $c_1=1/c_2=\rho^2$, hence $c_1c_2=1$ and thus is a subcase of (c). Precisely in this case $f$ is a generalized  eigenfunction of $\Delta$ with eigenvalue $-\rho^2$, a particular case of which is $\phi_0$.
\vspace{.15in}

(2) Given the similarity of the setting, it is not surprising that our line of argument sometimes goes near the study of real inverse Paley-Wiener theorems and the characterization of eigenfunctions of $\Delta$ mentioned earlier.  We pause briefly to point out the distinguishing features of our study.  In \cite{Andersen2004-1} Andersen considered real inverse Paley-wiener theorem characterizing functions in $L^2(X)$ whose Fourier transform is supported in  a ball around origin in $\R^+\times B$. It is clear from the proof of Theorem \ref{main-result} that only positive integral powers of $\Delta$ is required for this. As \cite{Andersen2004-1} is dealing with $L^2$-functions, Plancherel theorem has a crucial presence in the proof. But as explained in Section 3, this precludes the possibility of the support to degenerate and allow the function  to be an eigenfunction. On the other hand  aim of \cite{KRS2014, RS2014}  is to obtain a  characterization of the eigenfunction of $\Delta$. However the hypothesis of those theorems are strong enough  to determine the precise annulus around origin inside which $\wtilde{f}$ is supported.
\vspace{.2in}

(3) We recall that through the Iwasawa decomposition $G=NAK$, $X=G/K$ can be identified with the solvable Lie group $N\rtimes A$. Thus the rank one Riemannian symmetric spaces of noncompact type becomes a subclass  of Damek-Ricci spaces (known also as  $NA$ groups). Indeed  symmetric spaces are the most distinguished prototypes of $NA$ groups, even though they account for a very thin subcollection (see \cite{ADY}). In general a Damek-Ricci space is a Riemannian manifold and a solvable Lie group but not a symmetric space. To deal with a general Damek-Ricci space say $S$  one faces many fresh difficulties.  A major challenge is the absence of semisimple machinery which enters the picture through the $G$-action on $X=G/K$. A particular discomfort arises as we cannot decompose a function on $S$ in   $K$-types; a  very useful tool while working on symmetric spaces. The sense of {\em radiality} in these spaces is not connected with  group action.  Keeping this in mind we have completely avoided such  well-known techniques for symmetric spaces. The proof given here is thus  readily extendable to harmonic $NA$ groups. However   for Damek-Ricci spaces we have to make a compromise, as the characterization of $L^{2, \infty}$-eigenfunction as Poisson transform is still unavailable in the literature, albeit expected.  Precisely  (Theorem  \ref{main-result} (c)) `$f$ is a Possion transform' have to be substituted by a weaker statement `$f$ is  an eigenfunction of $\Delta$ with eigenvalue $-c_1$'.

\end{document}